\documentclass{article}
\usepackage{latexsym, amsmath, amssymb}
\usepackage{graphicx}
\usepackage{epstopdf}
\usepackage{amsmath, amsthm, amscd, amsfonts}
\usepackage{amsmath}
\usepackage{amssymb}
\usepackage{array}
\usepackage{multirow}
\usepackage{float}
\usepackage{enumerate}
\usepackage[a4paper, total={6in, 10in}]{geometry}
\usepackage{placeins}
\newtheorem{thm}{Theorem}[section]

\theoremstyle{definition} 
\theoremstyle{question} 
\theoremstyle{remark} 

\makeatletter
\newcommand*{\rom}[1]{\expandafter\@slowromancap\romannumeral #1@}
\makeatother

\def\bege{\begin{equation}} \def\ende{\end{equation}}

   \def\begr{\begin{eqnarray}}
	\def\endr{\end{eqnarray}} 
 
\def\bege{\begin{equation}} \def\ende{\end{equation}}
\def\begr{\begin{eqnarray}} \def\endr{\end{eqnarray}} \def\bnum{\begin{enumerate}} \def\enum{\end{enumerate}}
\begin{document}
	
		\begin{center}
		\textbf{ Edge Resolvability of Crystal Cubic Carbon Structure}
	\end{center}
	\begin{center}
	Sahil Sharma$^{1}$, Vijay Kumar Bhat$^{2,}$$^{\ast}$, Sohan Lal$^{3}$
\end{center}

\begin{center}
	School of Mathematics,\end{center} \begin{center}Shri Mata Vaishno
	Devi University,\end{center}\begin{center}Katra-$182320$, Jammu and
	Kashmir, India.
\end{center}

\begin{center}
	1. sahilsharma2634@gmail.com 2. vijaykumarbhat2000@yahoo.com 3. sohan1993sharma@gmail.com
\end{center}

\hspace{-5.0mm}\textbf{Abstract} Chemical graph theory is commonly used to analyse and comprehend chemical structures and networks, as well as their features. The resolvability parameters for graph $G$= $(V,E)$ are a relatively new advanced field in which the complete structure is built so that each vertex (atom) or edge (bond) represents a distinct position. In this article, we study the resolvability parameters i.e., edge resolvability of chemical graph of crystal structure of cubic carbon $CCS(n)$.\\\\

\hspace{-5.0mm}\textbf{Keywords}: Edge metric dimension, metric dimension, crystal cubic carbon, cubes

\hspace{-5.0mm}\textbf{$1$. Introduction }\\

     Mathematical applications have been increasingly popular in recent years. Graph theory has rapidly grown in theoretical conclusions as well as applicability to real-life issues as a handy tool for dealing with events relations. Distance is a concept that penetrates all of the graph theory, and it is utilized in isomorphism tests, graph operations, maximal and minimal connectivity problems, and diameter problems. Several characteristics linked to graph distances have piqued the interest of several scholars. One of them is the metric dimension.\\

	The notion of \textit{metric dimension} (MD) was introduced by Slater \cite{li}, Melter and Harary \cite{ki} independently. They termed the ordered subset $K =\{\alpha_{1},\alpha_{2},\alpha_{3},...,\alpha_{m}\}$ of vertices of graph $G = (V,E)$ as resolving set, if the representation $r(u|K) = (d(u,\alpha_1),d(u,\alpha_2),...,d(u,\alpha_m))$ for each $u\in V$ is unique, where $d(u,\alpha_{m})$ is the shortest path between $u$ and $\alpha_{m}$. The cardinality of the minimal resolving set is called MD (locating number) of $G$, and is denoted by $dim(G)$. Chartrand et al. \cite{ti} studied MD of path graphs $(P_n)$ and complete graph $(K_n)$. Later on, Caceres et al. \cite{pi} studied about MD of the cartesian product of  graphs, hypercubes, and cycles graphs. MD of regular graphs, Jahangir graphs, prism graphs, and circulant graphs was examined by several different authors. In chemistry, MD is used to give a series of compounds a unique representation, which leads to drug discovery. Other applications of MD include navigation, combinatorial optimization, and games such as coin weighing, etc. \cite{ti,VK}.\\
	
	 However in connection with the study of a new variant of metric generator in graphs Kelenc et al. \cite{len}, introduced \textit{edge metric dimension of a graph} (EMD). This is based on the fact that an ordered subset $H$ of vertices of $G =(V,E)$ is called an edge resolving set (edge metric generator), if for $e_1 \neq e_2 \in E$, there is a vertex $h \in H$ such that $d(e_1,h) \neq d(e_2,h)$, where $d(e,h) = min\{d(a,h),d(b,h)\}, ab = e \in E\; \mbox{and}\; h \in H$. The cardinality of the smallest edge metric generator (edge basis) is called EMD and is denoted by $edim(G)$. After that, Zhang and Gao  \cite{fd} discussed EMD of some convex polytopes. EMD of path graphs, complete graphs, bipartite graphs, wheel graphs, and many others have been studied. Recently, knor et al. \cite{mar} studied graphs with $edim(G) < dim(G)$.  After the introduction of EMD by Kelenc et al. \cite{len}, several other authors \cite{den,se,Kg,SS,sk}  have studied it thoroughly.\\
	 
	The structure of a chemical molecule is commonly described by a group of functional groups placed on a substructure. The structure is a graph-theoretic labelled graph, with vertex and edge labels representing atom and bond types, respectively. A collection of compounds is primarily defined by the substructure common to them, which is characterised by modifying the set of functional groups and/or permuting their locations. Traditionally, these “positions” simply reflect uniquely defined atoms (vertices) of the substructure (common subgraph). These positions seldom form a minimum set $K$ for which every two distinct vertices have distinct ordered k-tuples forming a minimum dimensional representation of the positions definable on the common subgraph. In this context, $K$ is referred to as a resolving set relative to V(G).\\
	 
	 We can assess whether any two compounds in the collection share the same functional group at a given position using the standard pattern. When analyzing whether a compound's characteristics are responsible for its pharmacological activity, this comparison statement is important \cite{pi,ti}. Sharma et al. \cite{CP} studied the mixed metric resolvability of chemical compound polycyclic aromatic hydrocarbon networks. The problem of computing the edge metric dimension of $CCS(n)$ and certain generalisations of these graphs is discussed in this work. One of the forms of carbon, namely diamond, is anticipated to convert into the $C_8$ structure, which is a cubic body-centered structure having 8 elements in its unit cell, at pressures exceeding 1000 GPa (gigapascal). Its composition is comparable to cubane and is found in one of silicon's metastable phases. Carbon sodalite was postulated as the structure of this phase in 2012. This cubic carbon phase could be useful in astronomy.\\

	Several authors studied different topological properties of $CCS(n)$. Yang et al. \cite{Yan} studied the Augmented Zagreb index, forgotten index, Balaban index, and redefined Zagreb indices of $CCS(n)$. Baig et al. \cite{BA} computed the degree-based additive topological indices mainly atom bond connectivity index, geometric index, first and second Zagreb index, etc. Further, Imran et al. \cite{IMA} studied eccentric based bond connectivity index and eccentric-based geometric arithmetic index of $CCS(n)$.  For further studies concerning the $CCS(n)$ one can refer to \cite{ZHA,ZAH} and references therein.\\
	 	 
	Zhang and Naeem \cite{Nem} studied the metric dimension of $CCS(n)$. They computed that for $n = 1$, $dim(CCS(n))$ is $3$, and for $n\geq 2$, $dim(CCS(n))$ is $ 7^{n-2}\times 16$. To the best of our knowledge, no work has been reported regarding the edge metric basis
	and edge metric dimension (EMD) of $CCS(n)$. For this, we consider the baseline paper reported by (Zhang and Naeem \cite{Nem}) and further extend the result with comparisons. In this study we compute the EMD of $CCS(n)$ and we also show that MD and EMD of $CCS(n)$ are the same.\\
	 	     
	\hspace{-5.0mm}\textbf{2. Crystal Cubic Carbon Structure $CCS(n)$}\\
	
	The valency of carbon allows it to create a variety of allotropes. Diamond and graphite are well-known forms of carbon, while crystal cubic carbon, also known as pcb, is one of its potential allotropes.
	  The chemical structure of $CCS(n)$ is constructed as
	\begin{enumerate}
		\item $CCS(n)$ starts from one unit cube, also called $CCS(1)$, and central cube as shown in Fig. 1.
		\item At the second level all the eight vertices of $CCS(1)$ are attached to different cubes through edges (bridge edges), which results in the construction of $CCS(2)$ as shown in Fig. 2.
		\item Again at the next level all the $7\times 8$ vertices of degree $3$ of $CCS(2)$ are attached to different cubes through edges, which results in the construction of $CCS(3)$.
		\item Similarly we construct $CCS(n)$ by attaching new cubes to the vertices of degree 3 at the preceding level as shown in Fig. 3.\\
		
	\end{enumerate} 
	At each level, we receive a new set of cubes that are attached to the vertex of degree 3 of cubes of the previous level. These cubes are referred to as the outermost cubes. In $CCS(2)$, we can also see that there are eight outermost cubes and $ 7\times 8$ vertices of degree 3. So at the third level,  $7\times 8$ cubes are attached. Similarly, there are $7^{n-2}\times8$ vertices of degree 3 at each subsequent level. So, this process is again repeated to get the next level.\\
	
	\hspace{-5.0mm}\textbf{2.1 Origin of Crystal Cubic Carbon Structure $CCS(n)$}\\
	
	Carbon, as one of the elements that have been known and used on the planet since ancient times, continues to draw attention due to its wide range of industrial and commercial applications, ranging from cutting tools to electrical and optoelectronic materials, and so on. Carbon exhibits the adaptability of $sp$, $sp^2$, and $sp^3$ hybridization states, resulting in a variety of rich allotropic carbon forms. Graphite and diamond are two prevalent allotropes of carbon, with the former being the most thermodynamically stable form of all known carbon allotropes at ambient temperatures (e.g. hexagonal diamond, amorphous carbon, carbyne, fullerenes, carbon nanotubes, crystal cubic carbon and graphene).\\
	
	Diamond is projected to convert into a body-centered cubic structure at ultra high pressures of above 1000 GPa. This phase is significant in astrophysics and the study of the deep cores of planets such as Uranus and Neptune. In 1979, \cite{Mat} a super dense and super hard material that looked like this phase was synthesised and published, having the Im3m space group and eight atoms per primitive unit cell. It was claimed that the so-called $C_8$ structure, with eight-carbon cubes comparable to cubane in the Im3m space group, had been synthesised, with eight atoms per primitive unit cell or 16 atoms per standard unit cell. The super cubane structure, the BC8 structure, a structure with clusters of four carbon atoms in tetrahedra in space group I43m with four atoms per primitive unit cell (eight per conventional unit cell), and a structure named "carbon sodalite" were all discussed in a report published in 2012, \cite{pok}. In 2017, Baig et al. \cite{BA} modified and extended the structure of carbon sodalite and named it crystal cubic carbon $CCS(n)$. $CCS(n)$ as a carbon allotrope is assumed to have wide a range of industrial applications. For more study on carbon allotropes one can refer \cite{ALL,CAL}.
	
		       \begin{center}
		\begin{figure}[h!]
			\centering
			\includegraphics[width=2.3in]{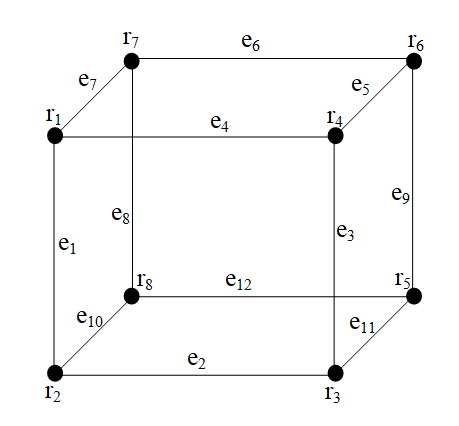}
			\caption{The chemical structure of $CCS(1)$.}\label{cu1}
		\end{figure}
	\end{center}
	
	\begin{center}
		\begin{figure}[h!]
			\centering
			\includegraphics[width=3.5in]{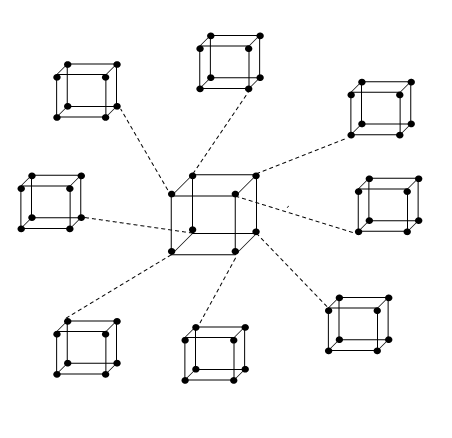}
			\caption{The chemical structure of $CCS(2)$.}\label{cu06}
		\end{figure}
	\end{center}

	\begin{center}
		\begin{figure}[h!]
			\centering
			\includegraphics[width=5.0in]{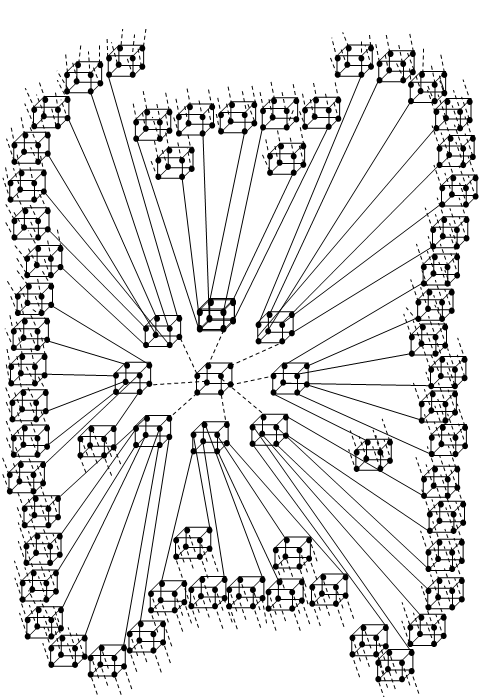}
			\caption{The chemical structure of $CCS(n)$.}\label{cube77}
		\end{figure}
	\end{center}
	\FloatBarrier	 
	 	 
\hspace{-5.0mm}\textbf{3. Edge metric dimension of $CCS(n)$}\\

The EMD of various graphs like Mobius network, wheel graph, convex polytopes, windmill graph, etc. have been studied. We will discuss the EMD of $CCS(n)$, $n\geq 1$ in this section.

\begin{thm}
	For $n=1$, the edge metric dimension of $CCS(n) = 3$.
	
\end{thm}

\begin{proof} First, claim that $edim(CCS(1)) \leq 3$. We have labelled all 8 vertices of the cube as $r_i$, $1\leq i \leq 8$, and edges as $e_i$, $1\leq i \leq 12$ as shown in Fig. 1.\\
Let $R_E = \{r_1,r_2,r_3\}$ be set of vertices of graph $CCS(1)$. We have to show that $R_E$ is an edge metric generator of $CCS(1)$.\\
	
	The representation of each edge $e_i$, $1\leq i \leq 12$ with respect to $R_E$ is given as 
	 \begin{align*}
		r\left[e_1|R_E\right] = (0,0,1) \hspace{3cm}& r\left[e_7|R_E\right] = (0,1,2)\\ 
		r\left[e_2|R_E\right] = (1,0,0)  \hspace{3cm}& r\left[e_8|R_E\right] = (1,1,2)\\ 
		r\left[e_3|R_E\right] = (1,1,0)  \hspace{3cm}& r\left[e_9|R_E\right] = (2,2,1)\\ 
		r\left[e_4|R_E\right] = (0,1,1)  \hspace{3cm}& r\left[e_{10}|R_E\right] = (1,0,1)\\ 
		r\left[e_5|R_E\right] = (1,2,1)  \hspace{3cm}& r\left[e_{11}|R_E\right] = (2,1,0)\\ 
		r\left[e_6|R_E\right] = (1,2,2)  \hspace{3cm}& r\left[e_{12}|R_E\right] = (2,1,1)\\ 
	\end{align*}
    The representation of each edge is unique. Therefore $edim(CCS(n)) \leq 3$.\\
    
    \hspace{-5.0mm}Next, we claim that $edim(CCS(1)) \geq 3$.\\
    Since, for any graph $G$, $edim(G) = 1$ iff $G$ is a path graph. This implies that $edim(CCS(1)) \neq 1$. Now, if $edim(CCS(1)) = 2$, then edge metric genarator $R_E$ of $CCS(1)$ consists of two vertices (say $R_E= \{r_i,r_j\}$; $i\neq j$). Due to the symmetry of $CCS(1)$, we have two choices for the vertices of $R_E$.
    \begin{enumerate}
    	\item Both the vertices  of $R_E$ are on the same face of the cube.\\
    	Further, we have two choices
    	\subitem (i) $r_i$ and $r_j$ are on the face diagonal of the cube.
    	\subitem (ii) $r_i$ and $r_j$ are on the same edge of the face of the cube.
    	\item Both the vertices of $R_E$ are on the primary diagonal of the cube.
    \end{enumerate}
     \textbf{Case 1}: Firstly, when both vertices are on the same face and along the face diagonal, then we consider $R_E = \{r_1,r_3\}$. We have $r\left[e_1|R_E\right] = (0,1) = r\left[e_4|R_E\right]$.\\
     Secondly, when both vertices are on the same face and along the same edge, then we consider  $R_E = \{r_2,r_3\}$. We have $r\left[e_4|R_E\right] = (1,1) = r\left[e_{12}|R_E\right]$.\\
     
     \textbf{Case 2}: When both vertices are on the primary diagonal of the cube, then we consider $R_E = \{r_3,r_7\}$. We have  $r\left[e_4|R_E\right] = (1,1) = r\left[e_5|R_E\right]$.\\
     
     In both cases, we get a contradiction. Therefore $edim(CCS(1)) \geq 3$. Hence $edim(CCS(1)) =3$.\\
     
	\end{proof}

\begin{thm}
	For $n\geq 2$, the edge metric dimension of $CCS(n) = 7^{n-2}\times 16$.
	
\end{thm}
\begin{proof}
	Let $R_E = \{r_1,r_2,r_3,...r_k\}$ be the set of vertices of type $a_1$ and $a_2$ of outermost cubes $U_n$ as shown in Fig. 4.\\
	The cube $U_n$ is the outermost and attached with the chain of cubes through  vertex $c$. The degree of vertex $c$ is $4$, and all other vertices are of degree 3. We assume that $R_E$ is an edge metric generator and $k =7^{n-2}\times 16$ (since there are $7^{n-2}\times 8$ cubes in the outermost layer of $CCS(n)$). The representation of any two distinct arbitrary edges of $CCS(n)$ can be compared in the following ways.\\
		\begin{center}
		\begin{figure}[h!]
			\centering
			\includegraphics[width=3.0in]{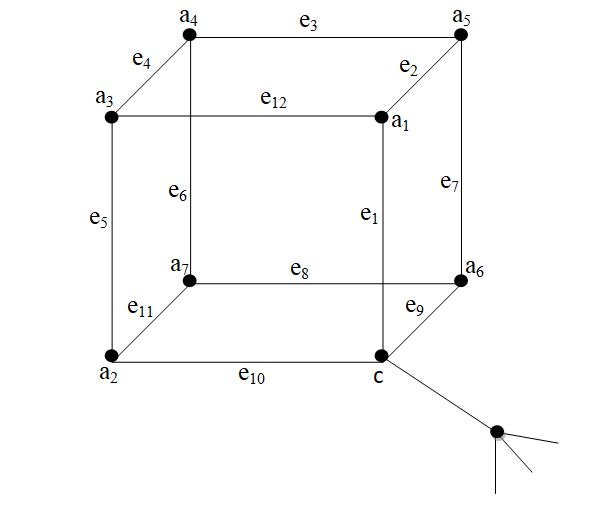}
			\caption{ Cube $U_n$ (arbitrary) from the outermost cubes of $CCS(n)$.}\label{cu2}
		\end{figure}
	\end{center}
	\FloatBarrier
	
	(i) When both the arbitrarily selected edges are on $U_n$ of $CCS(n)$ as shown in Fig. 4.\\
	
	We assume that $r_1 = a_1$ and $r_2 = a_2$, then $R_E = \{a_1,a_2,r_3,r_4...,r_k\}$. The representation of each edge $e_i$, $1\leq i \leq 12$ of cube $U_n$ concerning $R_E$ is given as\\

\hspace{-6.0mm}	$r\left[e_1|R_E\right] = \big(0,1,d(c,r_3),d(c,r_4),...d(c,r_k)\big)$\\
	$r\left[e_2|R_E\right] = \big(0,2,d(c,r_3)+1,d(c,r_4)+1,...d(c,r_k)+1\big)$\\
	$r\left[e_3|R_E\right] = \big(1,2,d(c,r_3)+2,d(c,r_4)+2,...d(c,r_k)+2\big)$\\
	$r\left[e_4|R_E\right] = \big(1,1,d(c,r_3)+2,d(c,r_4)+2,...d(c,r_k)+2\big)$\\
	$r\left[e_5|R_E\right] = \big(1,0,d(c,r_3)+1,d(c,r_4)+1,...d(c,r_k)+1\big)$\\
	$r\left[e_6|R_E\right] = \big(2,1,d(c,r_3)+2,d(c,r_4)+2,...d(c,r_k)+2\big)$\\
	$r\left[e_7|R_E\right] = \big(1,2,d(c,r_3)+1,d(c,r_4)+1,...d(c,r_k)+1\big)$\\
	$r\left[e_8|R_E\right] = \big(2,1,d(c,r_3)+1,d(c,r_4)+1,...d(c,r_k)+1\big)$\\
	$r\left[e_9|R_E\right] = \big(1,1,d(c,r_3),d(c,r_4),...d(c,r_k)\big)$\\
	$r\left[e_{10}|R_E\right] = \big(1,0,d(c,r_3),d(c,r_4),...d(c,r_k)\big)$\\
	$r\left[e_{11}|R_E\right] = \big(2,0,d(c,r_3)+1,d(c,r_4)+1,...d(c,r_k)+1\big)$\\
	$r\left[e_{12}|R_E\right] = \big(0,1,d(c,r_3)+1,d(c,r_4)+1,...d(c,r_k)+1\big)$\\
	
	We see that all these representations are different.\\
	
	(ii) When both the arbitrarily selected edges are on the distinct cubes of chain $CH_{c}$, one end of which is the cube of the outermost level as shown in Fig. 5.\\

	Let us assume that two selected edges $e_L$ and $e_M$ are on two distinct cubes namely $L$-cube and $M$-cube respectively on the chain of cubes. Assume that the chain of cubes $CH_{c}$ in $CCS(n)$, has one end as its central cube and the other end is the outermost cube which has a pair of distinct arbitrary edge resolving vertices (say $r_1$ and $r_2$). Since $e_L$ is one of the edges of $L$-cube and $e_M$ is edge of $M$-cube, then it clear that $d(e_L,R_E) < d(e_M,R_E)$. Hence representation of both edges concerning $R_E$ is distinct.\\\\
	
		(iii) When both the arbitrarily selected edges are bridge edges on the chain of cubes $CH_{c}$, one end of which is the cube of the outermost level as shown in Fig. 5.\\
	
	Let us assume that two selected edges $e_{b_i}$ and $e_{b_j}$ are on the chain of cubes. Assume that the chain of cubes $CH_{c}$ in $CCS(n)$, has one end as its central cube and the other end is the outermost cube which has a pair of distinct arbitrary edge resolving vertices (say $r_1$ and $r_2$).Then it clear that $d(e_{b_i},R_E) < d(e_{b_j},R_E)$. Hence representation of both edges concerning $R_E$ is distinct.\\\\
	
		\begin{center}
		\begin{figure}[H]
			\centering
			\includegraphics[width=4.0in]{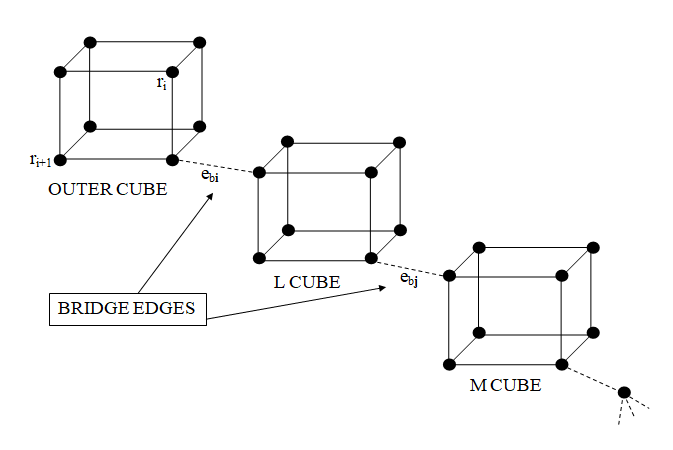}
			\caption{ Chain of cubes ($CH_{c}$) in $CCS(n)$.}\label{cu333}
		\end{figure}
	\end{center}
	\FloatBarrier

	(iv) When both the arbitrarily selected edges are on distinct chains of cubes (say $CH_{c_1}$ and $CH_{c_2}$) and these two chains are connected to a common cube called branching cube.\\\\
	Let us assume that the two  selected edges $e_M$ and $e_N$ are on two distinct cubes namely $M$-cube and $N$-cube respectively. Further, these cubes are on distinct chains of cubes (say $CH_{c_1}$ and $CH_{c_2}$) connected at branching $B$ cube as shown in Fig. 6.\\
	  Assume that the chain of cubes $CH_{c_1}$ and $CH_{c_2}$ in $CCS(n)$, has one end as branching cube and the other end as outermost cubes, which has a pair of distinct arbitrary edge resolving vertices $r_i = r_1, r_{i+1}$ and $r_j = r_3, r_{j+1}$. Also, the length of the shortest path between vertex $r_1$ and edge  $e_M$ on the $M$ cube is greater than the length of the shortest path between vertex $r_1$ and edge $e_N$ on the $N$ cube. This implies $d(e_N,r_1)\neq d(e_M,r_1)$, and hence $r(e_N,R_E) \neq r(e_M,R_E)$.
		\begin{center}
		\begin{figure}[h]
			\centering
			\includegraphics[width=3.5in]{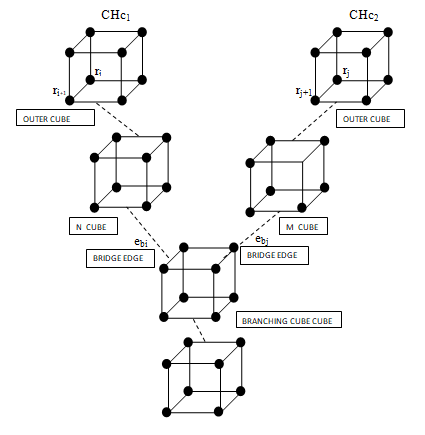}
			\caption{The distinct chains of cubes through branching cube}\label{cu444}
		\end{figure}
	\end{center}
	\FloatBarrier
	
	(v)  When both the arbitrarily selected bridge edges are on distinct chains of cubes (say $CH_{c_1}$ and $CH_{c_2}$) and these two chains are connected to a common cube called branching cube.\\\\
	Let us assume that the two  selected edges $e_{b_i}$ and $e_{b_j}$ are on two distinct chains of cubes namely $CH_{c_1}$  and $CH_{c_2}$ respectively. Further, these distinct chains are connected at branching $B$ cube as shown in Fig. 6.\\
	Assume that the chain of cubes $CH_{c_1}$ and $CH_{c_2}$ in $CCS(n)$, has one end as branching cube and the other end as outermost cubes, which has a pair of distinct arbitrary edge resolving vertices $r_i = r_1, r_{i+1}$ and $r_j = r_3, r_{j+1}$. Also, the length of the shortest path between vertex $r_1$ and edge  $e_{b_j}$ is greater than the length of the shortest path between vertex $r_1$ and edge $e_{b_i}$. This implies $d(e_{b_i},r_1)\neq d(e_{b_j},r_1)$, and hence $r(e_{b_i},R_E) \neq r(e_{b_j},R_E)$.\\\\
	
	(vi) When both the arbitrarily selected edges are on the central cube as shown in Fig. 7.
		\begin{center}
		\begin{figure}[h!]
			\centering
			\includegraphics[width=3.0in]{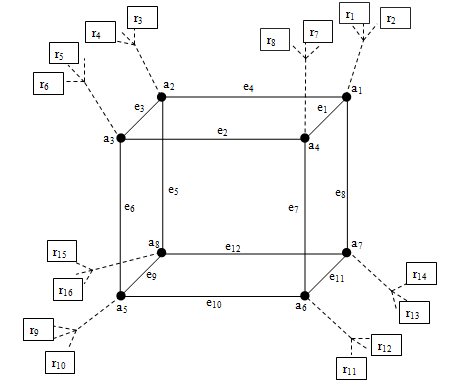}
			\caption{Central Cube of $CCS(n)$ and $r_i$,$r_{i+1}$ are vertices on the outermost level of cubes.}\label{cu5}
		\end{figure}
	\end{center}
	\FloatBarrier

	Assume that the two arbitrary selected edges are on the central cube. We have labelled all 8 vertices as $a_1$,$a_2$...$a_8$ and edges as $e_i$, $i = 1,2,3...,12$. Without loss of generality, we assume that vertices $r_{2i-1}, r_{2i}$, $i = 1,2,3...,8$, are on the outermost layer cubes in $CCS(n)$, and these outermost cubes having vertices $r_{2i-1}, r_{2i}$ are connected to the central cube at vertex $a_i$, $1\leq i\leq 8$, through a chain of cubes. Since each vertex is joined by three edges of the central cube ($a_1$ is joined by edge $e_1,e_4,e_8$) so the shortest distance of these edges from vertex $r_1$ and $r_2$ is the same but simultaneously edge $e_1$ is also joined by vertex $a_2$ but edges $e_4$ and $e_8$ are not joined by $a_2$. This further implies that the shortest path between edge $e_1$ and vertex $r_3$ is not the same as edges $e_4$ and $e_8$, and so on for other edges. Hence the representation is distinct for each edge concerning $R_E$.\\\\
	
	(vii) When both the arbitrarily selected edges are on the middle cube, which is neither the outer cube nor the central cube.
	 \begin{center}
		\begin{figure}[h!]
			\centering
			\includegraphics[width=4.5in]{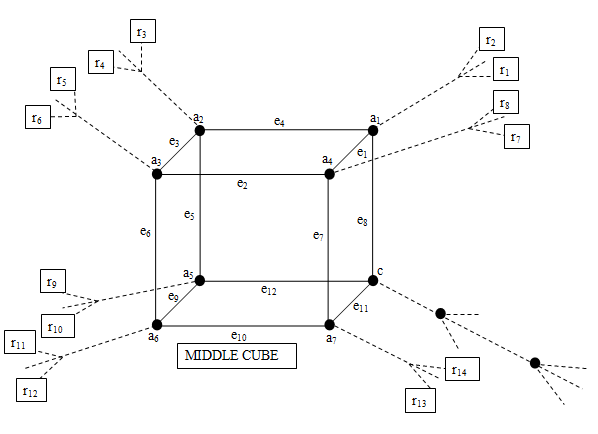}
			\caption{An arbitrary middle cube $C_m$ which is neither the outermost cube nor the central cube of $CCS(n)$.}\label{cu999}
		\end{figure}
	\end{center} 
	
	 Assume that the middle cube contains two arbitrary selected edges as shown in Fig. 8. We have labelled all 8 vertices as $c$,$a_1$,$a_2$...$a_7$ and edges as $e_i$, $i = 1,2,3...,12$.
	 We can assume that the vertices $r_1$ and $r_2$ lie on the cube (say $C_j$) of the outermost layer of $CCS(n)$ and that $C_j$ is joined to the middle cube $C_m$ at vertex $a_1$ by a chain of cubes without losing generality. Similarly, we can suppose that $r_{2i-1}, r_{2i}$ lie on the cubes of the outermost layer in $CCS(n)$, and these cubes are connected to middle cube $C_m$ at vertices $a_i$, $ 2\leq i \leq 7$, through the chain of cubes. As in case (vi), we find that each vertex is joined by three edges of cube $C_m$. Similar to the above case the representation is distinct for each edge concerning $R_E$.\\\\ 
	  
	 (viii) When one of the arbitrarily selected edges is on one of the cube, and the other is the bridge edge between cubes.\\
	 
	      Clearly, by the symmetry of the structure, the representation is distinct for each edge concerning $R_E$.\\
	      
	All these cases prove that $R_E = \{r_1,r_2,r_3...r_k\}$ is resolving set. Hence $edim(CCS(n))\leq 7^{n-2}\times 16$.\\\\
	
	Next, we claim that $edim\big(CCS(n)\big) \geq 7^{n-2} \times 16$. let $C_n$ be an arbitrary cube on the outermost layer of $CCS(n)$ as shown in Fig. 9. We have labelled all the vertices of $C_n$ as $a_i$, $i = 1,2,3...8$, and edges as $e_i$, $i = 1,2,3...,12$. The degree of each vertex is 3, except the vertex $a_1$, which is attached to the chain of the cube through an edge. Let $R_E = \{r_1,r_2,...r_k\}$ be an edge metric generator of $CCS(n)$. In order to complete the proof, we have to show that the edge metric generator $R_E$ has a minimum of two vertices from  $C_n$. To begin, we assume $v\in R_E$ such that $v\notin C_n$, then $r\left[e_1|R_E\right] = \big(d(a_1,r_1),d(a_1,r_2),d(a_1,r_3),...d(a_1,r_k)\big)$ = $r\left[e_{11}|R_E\right]$
	 which is a contradiction.\\
		\begin{center}
		\begin{figure}[h!]
			\centering
			\includegraphics[width=2.5in]{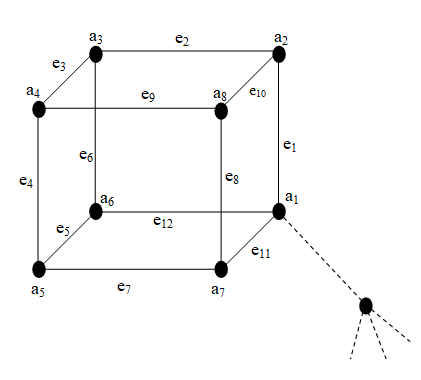}
			\caption{Outermost cube ($C_n$) of $CCS(n)$.}\label{cu888}
		\end{figure}
	\end{center}
	\FloatBarrier
	Secondly, we suppose that only one vertex of $R_E$ belongs to $C_n$. Without loss of generality, we assume that this common vertex is $r_1$.\\
	
	If $r_1 = a_1$, then\\
	$d(e_1|R_E) = \{0,d(a_1,r_2),d(a_1,r_3),...,d(a_1,r_k)\} = d(e_{11}|R_E)$, a contradiction.\\
	
	If $r_1 = a_2$, then\\
	$d(e_2|R_E) = \{0,d(a_1,r_2)+1,d(a_1,r_3)+1,...,d(a_1,r_k)+1\} = d(e_{10}|R_E)$, a contradiction.\\
	
	If $r_1 = a_3$, then\\
	$d(e_2|R_E) = \{0,d(a_1,r_2)+1,d(a_1,r_3)+1,...,d(a_1,r_k)+1\} = d(e_{6}|R_E)$, a contradiction.\\
	
	If $r_1 = a_4$, then\\
	$d(e_3|R_E) = \{0,d(a_1,r_2)+2,d(a_1,r_3)+2,...,d(a_1,r_k)+2\} = d(e_{9}|R_E)$, a contradiction.\\
	
	If $r_1 = a_5$, then\\
	$d(e_5|R_E) = \{0,d(a_1,r_2)+1,d(a_1,r_3)+1,...,d(a_1,r_k)+1\} = d(e_{7}|R_E)$, a contradiction.\\
	
	If $r_1 = a_6$, then\\
	$d(e_5|R_E) = \{0,d(a_1,r_2)+1,d(a_1,r_3)+1,...,d(a_1,r_k)+1\} = d(e_{6}|R_E)$, a contradiction.\\
	
	If $r_1 = a_7$, then\\
	$d(e_7|R_E) = \{0,d(a_1,r_2)+1,d(a_1,r_3)+1,...,d(a_1,r_k)+1\} = d(e_{8}|R_E)$, a contradiction.\\
	
	If $r_1 = a_8$, then\\
	$d(e_8|R_E) = \{0,d(a_1,r_2)+1,d(a_1,r_3)+1,...,d(a_1,r_k)+1\} = d(e_{10}|R_E)$, a contradiction.\\
	
	Our claim was backed by the inconsistency in all of the circumstances. As a result, the edge metric genarator $R_E$ of $CCS(n)$ has minimum of two vertices from the outermost cube $C_n$. The cubes in $CCS(n)$ are increased by a number equal to $7$ multiplied by the number of cubes in the outermost layer of the preceding level at each step or level, as can be seen from the construction of $CCS(n)$. This implies that $edim\big(CCS(n)\big) \geq 7^{n-2} \times 16$. Hence $edim\big(CCS(n)\big) = 7^{n-2} \times 16$.
	 
\end{proof}

  \hspace{-5.0mm}\textbf{4. Conclusion}\\	

In this article, the edge metric dimension of $CCS(n)$ has been studied. In particular, we have computed that for $n = 1$, $edim(CCS(n)) = 3$, and for $n\geq2$, $edim(CCS(n)) = 7^{n-2}\times 16$, which is similar to metric dimension of $CCS(n)$. These results are useful to study the structural properties of the chemical compound $CCS(n)$. In the future, we will extend our approach to find the fault-tolerant metric dimension, and the mixed metric dimension of $CCS(n)$.\\\\

 \hspace{-5.0mm}\textbf{5. Declarations}\\\\
 \hspace{-5.0mm}\textbf{Data availability statement}\\
 
 This article does not qualify for data sharing since no data sets were generated or analysed during this research.\\\\
 \hspace{-5.0mm}\textbf{Novelty statement}\\
 
 The edge resolvability of various well-known graphs has recently been computed. For instance, convex polytopes, the web graph, circulant graphs, certain chemical structures, and so on. In chemistry, one of the most important problems is to represent a series of chemical compounds mathematically so that each component has its representation. Graph invariants play an important role  to analyze the abstract structures of chemical graphs. But there are still some chemical graphs for which the vertex and edge resolvability has not been found yet, one such compound is crystal cubic carbon structure $CCS(n)$. Therefore, in this article, we compute the edge resolvability of $CCS(n)$, for $n\geq 1$ and see that $edim(CCS(n)) = dim(CCS(n))$.\\
 
 \hspace{-5.0mm}\textbf{Conflict of interest}\\
 
 The authors disclose that they have no competing interests.

\end{document}